\documentclass[11pt,reqno,british,a4paper,twoside]{article}

\usepackage{svn-multi}

\svnid{$Id: RHn-full.tex 304 2011-11-20 15:08:45Z swann $}

\usepackage[utf8]{inputenc}
\usepackage[sc]{mathpazo}
\usepackage{babel,mathtools,amsmath,amssymb,paralist,xspace,mathrsfs}
\usepackage{microtype,fixltx2e,fancyhdr,csquotes}

\usepackage[style=numeric,sorting=nyt,sortcites=true,backend=biber]{biblatex}
\usepackage[amsmath,thmmarks]{ntheorem}

\rmfamily

\mathtoolsset{mathic=true}
\setdefaultenum{\upshape (i)}{}{}{}

\renewcommand{\bibfont}{\small}

\theoremstyle{plain}
\theoremseparator{.}
\newtheorem{theorem}{Theorem}[section]
\newtheorem{proposition}[theorem]{Proposition}
\newtheorem{lemma}[theorem]{Lemma}
\newtheorem{corollary}[theorem]{Corollary}

\theorembodyfont{\normalfont}

\theoremheaderfont{\itshape}
\theoremsymbol{\begin{small}\ensuremath{\quad\triangle}\end{small}}

\theoremsymbol{\begin{small}\ensuremath{\quad\diamondsuit}\end{small}}

\theoremstyle{nonumberplain}
\theoremheaderfont{\itshape}
\theorembodyfont{\normalfont}
\theoremsymbol{\begin{small}\ensuremath{\quad\Box}\end{small}}
\newtheorem{proof}{Proof}

\qedsymbol{\begin{small}\ensuremath{\quad\Box}\end{small}}

\numberwithin{equation}{section}

\newenvironment{smallpmatrix}{\left(\begin{smallmatrix}}{\end{smallmatrix}\right)}

\newcommand{\LC}{\nabla}
\newcommand{\Nt}{\widetilde\nabla}
\newcommand{\Rt}{\widetilde R}

\newcommand{\T}{\mathscr{T}}

\DeclareMathOperator{\ad}{ad}
\DeclareMathOperator{\diag}{diag}
\DeclareMathOperator{\End}{End}
\DeclareMathOperator{\Id}{Id}

\newcommand{\norm}[1]{\lVert #1 \rVert}

\newcommand{\Lie}[1]{\operatorname{\textsl{#1}}}
\newcommand{\lie}[1]{\operatorname{\mathfrak{#1}}}

\newcommand{\hol}{\lie{hol}}

\newcommand{\Ort}{\Lie{O}}
\newcommand{\SO}{\Lie{SO}}
\newcommand{\so}{\lie{so}}

\newcommand{\la}{\lie{a}}
\newcommand{\f}{\lie{f}}
\newcommand{\g}{\lie{g}}
\newcommand{\tg}{\widetilde\g}
\newcommand{\h}{\lie{h}}
\newcommand{\lk}{\lie{k}}
\newcommand{\m}{\lie{m}}
\newcommand{\n}{\lie{n}}
\newcommand{\ls}{\lie{s}}

\newcommand{\bR}{{\mathbb R}}

\DeclareMathOperator{\RH}{\bR H}

\newcommand{\eqbreak}{\\&\qquad}

\makeatletter
\newcommand{\mythanks}{\xdef\@thefnmark{}\@footnotetext}
\makeatother

\begin{document}
\thispagestyle{empty}

\begin{center}
  \LARGE\bfseries The homogeneous geometries of real hyperbolic space
\end{center}
\begin{center}
  \Large M. Castrillón López, P. M. Gadea and A. F. Swann
\end{center}

\mythanks{Partially supported by MICINN, Spain, under grant
MTM2011-22528 and by the Danish Council for Independent Research,
Natural Sciences}


\begin{abstract}
  We describe the holonomy algebras of all canonical connections of
  homogeneous structures on real hyperbolic spaces in all dimensions.
  The structural results obtained then lead to a determination of the
  types, in the sense of Tricerri and Vanhecke, of the corresponding
  homogeneous tensors.  We use our analysis to show that the moduli
  space of homogeneous structures on real hyperbolic space has two
  connected components.
\end{abstract}

\bigskip
\begin{center}
  \begin{minipage}{0.8\linewidth}
    \begin{small}
      \tableofcontents
    \end{small}
  \end{minipage}
\end{center}

\bigskip
\begin{small}\noindent
\end{small}

\section{Introduction}
\label{sec:introduction}

Homogeneous manifolds provide a rich and varied class of spaces on
which to study Riemannian geometry.  One difficulty that arises is
that the same Riemannian manifold \( (M,g) \) can admit several
different descriptions as a homogeneous space~\( G/H \).  It is
surprising how little is known this problem for many well-known
spaces.

A substantial attempt to solve this problem was made by
\textcite{Ambrose-S:homogeneous}.  They characterised the property
that \( (M,g) \) is homogeneous in terms of the existence of a tensor
\( S \) satisfying a certain set of non-linear differential equations.
Each homogeneous description of \( (M,g) \) gives rise to a different
solution to these equations.  These equations were studied further by
\textcite{Tricerri-V:homogeneous}, who introduced a decomposition of
\( S \) into components under the action of the orthogonal group, and
produced a number of examples illustrating the occurrence of different
possible classes.  In particular, they showed that in dimension \( 3
\), the real hyperbolic space \( \RH(3) \) admits homogeneous tensors
of two different types.  However, they left as an open problem, the
determination of all homogeneous tensors on \( \RH(n) \) for \(
n\geqslant 3 \) \cite[p.~55]{Tricerri-V:homogeneous}.

In \cite{Castrillon-Lopez--GS:RHn}, we took a different route and used
general results of \textcite{Witte:cocompact} on co-compact subgroups
to determine all the groups acting transitively on~\( \RH(n) \).  This
left open the determination of the corresponding homogeneous
tensors~\( S \) and their types.  Any homogeneous space \( G/H \) with
a Lie algebra decomposition \( \g = \h + \m \) carries a canonical
connection~\( \Nt \), characterised by the property that \( G
\)-invariant tensors are parallel for~\( \Nt \).  By work of
Nomizu~\cite{Nomizu:affine}, the tensor \( S \) depends only on the
holonomy algebra~\( \hol \leqslant \h \) of~\( \Nt \) and \( \hol + \m
\) determines the Lie algebra of a smaller group acting transitively
on \( G/H \).

In this paper we answer two questions: what are the holonomy algebras
of the canonical connections on \( \RH(n) \)? and what are the types
of the corresponding homogeneous tensors?  Regarding a geometric
structure as being given by a collection of tensors that are parallel
with respect to some connection, the answer to the first question thus
determines which geometric structures may be realised homogeneously on
\( \RH(n) \).  Our answer to the first question is given by:

\begin{theorem}
  \label{thm:holm}
  The holonomy algebras of canonical connections on \( \RH(n) \) are
  \( \so(n) \) and all the reductive algebras
  \begin{equation*}
    \lk = \lk_0 + \lk_{ss} 
  \end{equation*}
  of compact type with \( \lk_0 \cong \bR^r \) Abelian and \( \lk_{ss}
  \) semi-simple such that
  \begin{equation*}
    3r + \dim\lk_{ss} \leqslant n-1.
  \end{equation*}
\end{theorem}

The proof includes a description of how this algebra acts on the
tangent space of \( \RH(n) \).  We then use these results to determine
the complete answer to the second question, extending the partial
results of
\cite{Tricerri-V:homogeneous,Pastore:T13,Pastore:T13rep,Pastore:canonical,Pastore-V:T1-T2}.
Furthermore, the ideas of our constructions are used to show that the
moduli space of homogeneous structures on \( \RH(n) \), \( n>1 \),
with fixed scalar curvature has exactly two components.

The paper is organised as follows.  In Section\nobreakspace \ref {sec:ambrose-singer}, we
briefly recall the results of Ambrose \& Singer and Nomizu relating
homogeneous spaces to homogeneous tensors.  We then specialise to the
real hyperbolic space in~Section\nobreakspace \ref {sec:homogeneous} and review our result
on the groups that act transitively, establishing notation for the
rest of the paper.  The determination of the holonomy algebras and
their isotropy representations is given
in~Section\nobreakspace \ref {sec:holonomy-algebras}.  We use this in
Section\nobreakspace \ref {sec:homogeneous-tensors} to determine the homogeneous tensors
and their types.  Finally, we use our results to determine the
connected components of the moduli space of homogeneous tensors on \(
\RH(n) \) and discuss a couple of geometric consequences
in~Section\nobreakspace \ref {sec:consequences}.

\section{The Ambrose-Singer equations}
\label{sec:ambrose-singer}

Let \( (M,g) \) be a connected, simply-connected complete Riemannian
manifold.  Suppose \( S \) is a tensor of type \( (1,2) \); so for
each \( X \in TM \), we have that \( S_X \)~is an endomorphism of~\(
TM \).  Writing \( \LC \) for the Levi-Civita connection of \( g \),
we define a new connection \( \Nt = \LC - S \).  In general, \( \Nt \)
has non-zero torsion.  \textcite{Ambrose-S:homogeneous} showed that \(
(M,g) \) admits a homogeneous structure if and only if there is an~\(
S \) such that
\begin{equation}
  \label{eq:AS}
  \Nt g = 0,\quad \Nt R = 0\quad\text{and}\quad \Nt S = 0,
\end{equation}
where \( R \) is the curvature tensor of~\( \LC \).
\textcite{Nomizu:affine} gave this homogeneous description as follows.
Fix a point \( p \) in \( M \).  The holonomy algebra \( \hol \) is
the subalgebra of the endomorphisms of \( T_pM \) generated by the
elements \( \{\,\Rt_{X,Y}: X,Y\in T_pM \,\} \), where \( \Rt \) is the
curvature of~\( \Nt \).  Writing \( \m = T_pM \), the vector space
\begin{equation*}
  \tg = \hol + \m
\end{equation*}
has a Lie bracket defined by
\begin{equation*}
  [U,V] = UV - VU,\quad [U,X] = U(X), \quad [X,Y] = \Rt_{X,Y} + (S_XY
  - S_YX)  
\end{equation*}
for \( U,V \in \hol \subset \End(T_pM) \) and \( X,Y \in \m = T_pM \).
Exponentiating these groups we obtain a reductive homogeneous
description of~\( M \) as \( \widetilde G/H \), where \( \widetilde G
\) and \( H \) have Lie algebras \( \tg \) and \( \hol \)
respectively.  The connection \( \Nt \) is now the canonical
connection of the reductive space \( (\widetilde G/H,\m) \).  Indeed
for any homogeneous space \( M = G/H \) with reductive description \(
\g = \h + \m \), the canonical connection is given at the identity by
\begin{equation}
  \label{eq:canonical}
  \Nt_BC = - [B,C]_{\m}
\end{equation}
where \( C \in \g \) is extended to the vector field on \( M \) whose
one-parameter group is \( gH\mapsto \exp(tC)gH \).  The canonical
connection has the property that every left-invariant tensor on~\( M
\) is parallel.

\section{Homogeneous descriptions of real hyperbolic space}
\label{sec:homogeneous}

The description of \( \RH(n) \) as a symmetric space is
\begin{equation*}
  \RH(n) = \SO(n,1)/\Ort(n),
\end{equation*}
where we take \( \SO(n,1) \) to be the set of matrices of determinant
\( +1 \) preserving the form \( \diag(\Id_{n-1},
\begin{smallpmatrix}
  0&1\\1&0
\end{smallpmatrix}
) \).  The connected isometry group has Iwasawa decomposition \(
\SO_0(n,1) = \SO(n)\bR_{>0} N \) whose Lie algebra is \( \so(n,1) = \so(n)
+ \la + \n \), given concretely by
\begin{gather*}
  \so(n) = \left\{
    \begin{smallpmatrix}
      B&v&v\\
      -v^T&0&0\\
      -v^T&0&0
    \end{smallpmatrix}
    : B\in \so(n-1),\ v\in \bR^{n-1}
  \right\},  \\
  \la = \bR\diag(0,\dots,0,1,-1), \quad \lie n = \left\{
    \begin{smallpmatrix}
      0&0&v\\
      -v^T&0&0\\
      0&0&0
    \end{smallpmatrix}:v\in\bR^{n-1} \right\}.
\end{gather*}

If \( G \) acts transitively on \( \RH(n) \) then \( G \backslash
\RH(n) \) is a point, so compact.  It follows that \( G \backslash
\SO_0(n,1) \) is an orbit space of the compact group \( \Ort(n) \),
thus \( G \) is a non-discrete co-compact subgroup of \( \SO_0(n,1)
\).  Witte's structure theory for co-compact groups
\cite{Witte:cocompact} then leads to the following result.

\begin{theorem}[\cite{Castrillon-Lopez--GS:RHn}]
  \label{thm:structure}
  The connected groups acting transitively on $\RH(n)$ are the
  connected isometry group $\SO_0(n,1)$ and the groups $G = F_rN$,
  where $N$ is the nilpotent factor in the Iwasawa decomposition of
  $\SO(n,1)$ and $F_r$ is a connected closed subgroup of
  $\SO(n-1)\bR_{>0}$ with non-trivial projection to $\bR_{>0}$.  \qed
\end{theorem}

The case \( F_r = \bR_{>0} \), gives the description \( \RH(n) =
\bR_{>0} N \) of real hyperbolic space as a solvable group.

\section{The holonomy algebras}
\label{sec:holonomy-algebras}

Assume that \( G = F_rN \) acts transitively on \( \RH(n) \) as in
Theorem\nobreakspace \ref {thm:structure}.  This implies that \( \RH(n) = G/H \), with \( H
= F_r\cap \SO(n-1) \).  We have immediately that \( H \) is reductive,
and thus
\begin{equation*}
  \h = \h_0 + \h_{ss},
\end{equation*}
where \( \h_0 \) is Abelian and \( \h_{ss} \) is semi-simple.  Let us
write
\begin{equation*}
  \f_r = \h + \la_r,\qquad
  \g = \h + \la_r + \n,
\end{equation*}
with \( \la_r \) projecting non-trivially to \( \la =
\operatorname{Lie}\bR_{>0} \).  Since \( \f_r \) is a subalgebra of \(
\so(n-1) \oplus \la \), it admits a positive definite invariant
metric.  This implies that \( \f_r \) is reductive with
\begin{equation*}
  \f_r = (\h_0 + \la_r) + \h_{ss}.
\end{equation*}
In particular, \( [\la_r,\h] = 0 \) and \( \dim \la_r = 1 \).

Let us write
\begin{equation*}
  \ls = \la + \n,\qquad \ls_r = \la_r + \n,
\end{equation*}
and note that \( [\ls,\ls] = \n = [\ls_r,\ls_r] \).  For later use, we
remark that \( \la_r \) is not canonically specified, but is any
one-dimensional complement to \( \h_0 + \n \) in
\begin{equation*}
  \ls_f = (\f_r)_0 + \n.
\end{equation*}

A homogeneous Riemannian structure on \( G/H \) depends on a choice of
\( \ad_H \)-invariant complement \( \m \) to \( \h \) in \( \g \).
Such a complement is the graph of an \( \h \)-equivariant map
\begin{equation}
  \label{eq:phi-r}
  \varphi_r\colon \ls_r \to \h.
\end{equation}

Choose an \( \h \)-equivariant map \( \chi_r\colon \ls \to \ls_r \)
extending the identity on \( \n \).  Define \( \varphi\colon \ls \to
\h \) as
\begin{equation}
  \label{eq:phi}
  \varphi = \varphi_r \circ \chi_r.
\end{equation}

\begin{proposition}
  \label{prop:hol1}
  The Lie algebra \( \hol \) of the holonomy group of the canonical
  connection \( \Nt \) associated to the decomposition \( \g = \h + \m
  \), is
  \begin{equation*}
    \hol = \varphi_r(\n) = \varphi(\n).
  \end{equation*}
\end{proposition}

\begin{proof}
  The holonomy algebra is spanned by \( [\m,\m]_{\h} \).  For \(
  A\in\la \) the standard generator and arbitrary \( X\in\n \), we
  have \( [A,X]=X \).  The space \( \m \) is spanned by
  \begin{equation*}
    X_\varphi \coloneqq
    X + \varphi X \in \m,\qquad\text{for \( X \in \n \),}
  \end{equation*}
  and the element
  \begin{equation*}
    \xi \coloneqq \chi_r A + \varphi A \eqqcolon A + A_0.
  \end{equation*}
  Noting that \( [A_0,\f_r] = 0 \), we compute
  \begin{equation*}
    \begin{split}
      [\xi,X_\varphi]
      &= [A,X] + [A,\varphi X] + [A_0,X] + [A_0,\varphi X]\\
      &= X + 0 + [A_0,X] + 0.
    \end{split}
  \end{equation*}
  This element lies in \( \n \).  Moreover, \( A_0 \) acts on \( \n \)
  as an element of \( \so(n-1) \) on \( \bR^{n-1} \), in particular
  its characteristic polynomial has no non-zero real roots.  This
  implies that \( 1 + \ad(A_0) \colon \n \to \n \) is invertible and
  so \( \{ [\xi,X_\varphi]:X \in \n \} \) spans \( \n \).  For \( Y
  \in \n \subset \h + \m \), we have \( Y_{\m} = Y + \varphi(Y) \) and
  so \( Y_{\h} = - \varphi(Y) \).  We conclude that \( \hol \)
  contains \( \{ - \varphi[\xi,X_\varphi] : X \in \n \} = \varphi(\n)
  \).

  For \( X,Y\in\n \), we have
  \begin{equation*}
    \begin{split}
      [X_\varphi,Y_\varphi]
      &= [X,Y] + [X,\varphi Y] + [\varphi X,Y] + [\varphi X,\varphi Y]\\
      &= 0 + ([X,\varphi Y] + [\varphi X,Y]) + [\varphi X,\varphi Y].
    \end{split}
  \end{equation*}
  The last term lies in \( \h \), whereas the middle pair lies in \(
  \n \).  Projecting to \( \h \subset \h + \m \), the middle pair
  contributes \( -2[\varphi X,\varphi Y] \), since \( \varphi \) is \(
  \h \)-equivariant.  Thus
  \begin{equation*}
    [X_\varphi,Y_\varphi]_{\h} = - [\varphi X,\varphi Y].
  \end{equation*}
  We find that
  \begin{equation*}
    \hol = \varphi(\n) + [\varphi(\n),\varphi(\n)].
  \end{equation*}
  But \( \h \) is reductive and \( \varphi(\n) \) is a sum of \( \h
  \)-modules so \( [\varphi(\n),\varphi(\n)] \subset [\h,\varphi(\n)]
  \subset \varphi(\n) \) and \( \hol = \varphi(\n) \) as claimed.
\end{proof}

We are now ready to prove Theorem\nobreakspace \ref {thm:holm}.  

\begin{proof}[of Theorem\nobreakspace \ref {thm:holm}]
  Let us first show that the holonomy algebra has the claimed form.
  Via \( \varphi \) we have that the \( \h \)-module \( \hol \) is
  isomorphic to a submodule~\( V_{\hol} \) of \( \n\cong\bR^{n-1} \).
  Write \( \lk = \hol \) and note that \( \lk \) is a subalgebra of \(
  \so(n-1) \), so of compact type.  We may thus split
  \begin{equation*}
    \lk = \lk_0 + \lk_{ss} 
  \end{equation*}
  as a sum of Abelian and semi-simple algebras.  This gives a similar
  splitting \( V_{\hol} = V_0 + V_{ss} \).

  Now \( \lk \) acts effectively on \( \m \cong \la_r + \n \), and
  trivially on \( \la_r \), and its action preserves the inner product
  on \( \n \).  The action of \( \lk_{ss} \) is effective on \( V_{ss}
  \) and trivial on \( V_0 \).  The action of \( \lk_0 \) is trivial
  on all of \( V_{\hol} \).

  As \( \lk_0 \cong \bR^r \) is Abelian, its irreducible metric
  representations are direct sums of modules of real dimension \( 2
  \), and an effective representation is of dimension at least \( 2r
  \).  Thus \( \n \) contains inequivalent modules of dimension \(
  \dim V_{\hol} = \dim\lk_{ss} + r \) and \( 2r \).  It follows that
  \( n-1 = \dim\n \geqslant \dim\lk_{ss} + 3r \).

  Conversely, given a reductive algebra \( \lk = \lk_0 + \lk_{ss} \)
  of compact type satisfying this constraint on dimensions, we wish to
  show that it arises a holonomy algebra for a canonical connection.

  Let \( V_{\lk} \) be a copy of the \( \lk \)-module \( \lk \) and
  let \( V_1 \) be a minimal effective metric representation of~\(
  \lk_0 \cong \bR^r \).  Then \( \dim V_1 = 2r \) and we put
  \begin{equation*}
    \n = \bR^{n-1} = V_{\lk} + V_1 + \bR^m ,
  \end{equation*}
  with \( \bR^m \) a trivial \( \lk \)-module.  This decomposition of
  \( \bR^{n-1} = \n \) admits a \( \lk \)-invariant inner product
  extending a bi-invariant metric on \( \lk \cong V_{\lk} \) and the
  invariant inner product on \( V_1 \).  Such an inner product
  realises \( \lk \) as a subalgebra of \( \so(n-1) \).  Let \(
  \psi\colon V_{\lk} \to \lk \) be an isomorphism of \( \lk
  \)-modules.  Defining \( \varphi \) to be \( \psi \) on \( V_{\lk}
  \) and zero on \( V_1+\bR^r+\la \) then realises \( \lk \) as the
  holonomy algebra of a canonical connection~\( \Nt \) on \( \RH(n) \)
  with \( \g = \lk + \la + \n \).
\end{proof}

Note that in the construction of the second part of the proof, the Lie
algebra \( \lk \) exponentiates to a closed (so compact) subgroup~\( K
\) of \( \SO(n-1) \), and so \( \lk \) is the isotropy algebra of a
homogeneous realisation of \( \RH(n) \).  Also note that the module \(
V_1 + \bR^r \) may be replaced by any metric representation of \( \lk
\) on which \( \lk_0 \) acts effectively, but in this case the
corresponding subgroup of \( \SO(n-1) \) may not be closed.

\section{Homogeneous tensors}
\label{sec:homogeneous-tensors}

We now wish to compute the homogeneous tensor \( S = \LC - \Nt \)
associated to a invariant Riemannian structure on~\( G/H \).

Let \( g \) be the Riemannian metric and let \( g \) also denote its
restriction to~\( \m \).  This bilinear form on \( \m \) is \( \ad_H
\)-invariant.  At \( eH \), the homogeneous tensor is given by
\begin{equation}
  \label{eq:S}
  2g(S_BC,D) = g([B,C],D) - g([C,D],B) + g([D,B],C),
\end{equation}
for \( B,C,D\in \m \).  This follows from \eqref{eq:canonical} and
\cite[p.~183]{Besse:Einstein}.  The description of \( \RH(n) \) as a
symmetric space corresponds to \( S\equiv 0 \).  We thus concentrate
on the other homogeneous descriptions associated to subgroups \( F_r
\) of \( \SO(n)\bR_{>0} \), and use the notation of the previous
section.

Note that \( g \) induces \( \h \)-invariant inner products \( g_r =
(1+\varphi_r)^*g \) on \( \ls_r \) and \( g_\varphi = \chi_r^*g_r =
(\chi_r^*+\varphi^*)g \) on \( \ls \).  As we remarked above, the
space \( \la_r \) is not canonical.  The module \( \ls \) splits \(
g_\varphi \)-orthogonally as a sum of a trivial \( \h \)-module \(
\ls_0 \) and a module \( \ls_1 \subseteq \n \) that decomposes as a
sum of non-trivial \( \h \)-modules.  The space \( \la \) is contained
in \( \ls_0 \) and is any complement to \( \ls_0\cap\n \).  In
particular, we can take \( \la \) to be \( g_\varphi \) orthogonal to
\( \n \) and take \( \la_r = \chi_r\la \).

As above, let \( A \) be the generator of \( \la \) that satisfies \(
\ad(A)|_{\n} = +1 \).  An arbitrary element~\( B \) of \( \m \) may be
written as
\begin{equation*}
  B = \lambda_B \xi + N_B,
\end{equation*}
where \( N_B = (X_B)_\varphi \), for some \( X_B\in \n \).  By our
choice of \( \la \), we see that
\begin{equation*}
  \lambda_B = g(B,\xi)/g(\xi,\xi).
\end{equation*}

\begin{lemma}
  Let \( S \) be a homogeneous tensor on \( \RH(n) \) associated to
  module maps \( \varphi \) and \( \varphi_r \) as
  in \eqref{eq:phi} and~\eqref{eq:phi-r}.  Then
  \begin{equation}
    \label{eq:Sc1}
    \begin{split}
      g(S_BC,D) = -\lambda_C&g(B,D) + \lambda_Dg(B,C) + g([B',C],D) \\
      + &\tfrac12\bigl( \lambda_B (h_r(C,D) - h_r(D,C)) \eqbreak -
      \lambda_C (h_r(B,D) + h_r(D,B)) \eqbreak + \lambda_D (h_r(B,C) +
      h_r(C,B)) \bigr),
    \end{split}
  \end{equation}
  where \( B' = \varphi(\lambda_B A + X_B) = \varphi(B_{\ls}) \in \h
  \) and \( h_r(B,C) = g([A_1,X_B]_\varphi,C) \), \( A_1 = \chi_rA - A
  \in \so(n-1) \).
\end{lemma}

\begin{proof}
  To see this, let us compute
  \begin{equation*}
    \begin{split}
      [B,C]_{\m} &= (\lambda_B[\xi,N_C]- \lambda_C[\xi,N_B]+[N_B,N_C])_{\m}\\
      &= \lambda_B(N_C + [\varphi A,X_C] + [A_1,X_C]_\varphi) \eqbreak
      - \lambda_C(N_B + [\varphi A,X_B] + [A_1,X_B]_\varphi) \eqbreak
      + [\varphi X_B,N_C] - [\varphi X_C,N_B],
    \end{split}
  \end{equation*}
  where we have used that \( \varphi[\varphi A,X_C] = [\varphi
  A,\varphi X_C] = 0 \).  This gives
  \begin{equation*}
    \begin{split}
      g([B,C],D) &= \lambda_B\bigl( g(C,D) - \lambda_C g(\xi,D) +
      g([\varphi A,C],D) + h_r(C,D) \bigr) \eqbreak - \lambda_C\bigl(
      g(B,D) - \lambda_B g(\xi,D) + g([\varphi A,B],D) + h_r(B,D)
      \bigr) \eqbreak + g([\varphi
      X_B,C],D) - g([\varphi X_C,B],D) \\
      &= \lambda_B g(C,D) - \lambda_C g(B,D) + g([B',C],D) -
      g([C',B],D) \eqbreak + \lambda_B h_r(C,D) -
      \lambda_C h_r(B,D)\\
    \end{split}
  \end{equation*}
  and the result~\eqref{eq:Sc1} follows from~\eqref{eq:S} and the fact
  that \( g \) is \( \h \)-invariant.
\end{proof}

We now wish to determine the possible types of \( S \) in the sense of
\textcite{Tricerri-V:homogeneous}.  The first of the Ambrose-Singer
equations~\eqref{eq:AS} implies that at each point \( S_x \) preserves
\( g \), so \( S \) is a section of \( T^*M \otimes \so(n) \cong TM
\otimes \Lambda^2TM \).  As representation of \( \so(n) \), this
space decomposes as
\begin{equation*}
  \Gamma(TM \otimes \Lambda^2TM) \cong \T_1 \oplus \T_2 \oplus \T_3
\end{equation*}
with \( \T_1 \cong \Gamma(TM) \) and \( \T_3 \cong \Gamma(\Lambda^3TM)
\).  One says that \( S \) is of type \( \T_i \) if \( S \) lies in \(
\T_1 \), and correspondingly \( S \) is of type \( \T_{i+j} \) if \( S
\in \T_i + \T_j \).  

\textcite{Tricerri-V:homogeneous} showed that if \( (M,g) \) is
connected, simply-connected and complete, then it admits a homogeneous
structure of type \( \T_1 \) if and only if \( (M,g) \) is isometric
to the standard metric on \( \RH(n) \).  The corresponding homogeneous
description is that of \( \RH(n) \) as a solvable group.  Furthermore,
\cite{Pastore:T13rep} showed that structures on \( \RH(n) \) of type
\( \T_{1+3} \) correspond to semi-simple isotropy groups.  We can now
describe all the types of homogeneous structures on \( \RH(n) \).

\begin{theorem}
  \label{thm:type}
  Let \( S \) be a non-zero homogeneous tensor for~\( \RH(n) \) with
  holonomy algebra \( \hol \).  Then \( S \) always has a non-trivial
  component in~\( \T_1 \) and \( S \)~is of type \( \T_1 \) if and
  only if \( \hol \) is~\( 0 \).

  The structure is of strict type \( \T_{1+3} \) if and only if \( \la
  \subset \ker\varphi \) and \( \hol \) is a non-zero semi-simple
  algebra acting trivially on \( \ker\varphi \), in the notation of
  Section\nobreakspace \ref {sec:holonomy-algebras}.

  Otherwise \( S \) is of general type.
\end{theorem}

\begin{proof}
  From~\eqref{eq:Sc1}, we have \( S = S^1+S^2 \), with
  \begin{equation}
\label{eq:S1}
    S^1_BC = g(\xi,\xi)^{-1}(g(B,C)\xi - g(C,\xi)B),
  \end{equation}
  which is of type \( \T_1 \), and
  \begin{gather*}
    S^2_BC = [B',C] + S^r_BC,\\
    \begin{split}
      2 S^r_BC &= (\lambda_B (Z_r - Z_r^*)(X_C) - \lambda_C (Z_r +
      Z_r^*)(X_B))_\varphi \eqbreak +
      (h_r(B,C)+h_r(C,B))g(\xi,\xi)^{-1}\xi,
    \end{split}
  \end{gather*}
  where \( Z_r\colon \n \to \n \) is \( Z_r = \ad(A_1)|_{\n} \).

  We claim that \( S^2 \) is of type \( \T_{2+3} \).  This means that
  \( \sum_{i=1}^n S^2_{e_i}e_i = 0 \) for an orthonormal basis of \(
  \m \).  Noting that this condition is independent of the choice
  of orthonormal basis, we deal with the two terms of \( S^2 \)
  separately.

  Let us show that the \( (1,2) \)-trace \( \sum_{i=1}^n S^r_{e_i}e_i
  \) is zero.  Write \( g_0 \) for the metric on \( \n \) preserved by
  \( \so(n-1) \); this metric is unique up to scale.  Then \( Z_r \)
  is skew-adjoint with respect to \( g_0 \).  Let \( E_1,\dots,E_{n-1}
  \) be a \( g_\varphi \)-orthonormal basis diagonalising~\( g_0 \),
  so \( g_0(E_i,E_i) = t_i > 0 \).  Then the matrix \( (z_{ij}) \) of
  \( Z_r \) satisfies \( t_iz_{ji} + t_jz_{ij} = 0 \) so \( z_{ii} = 0
  \).  Putting \( e_i = (E_i)_\varphi \) and \( e_n =
  \xi/g(\xi,\xi)^{1/2} \), we obtain an orthonormal basis for all of
  \( \m \).  For \( i=1,\dots,n-1 \), we have that \( S^r_{e_i}e_i
  \) is \( g(\xi,\xi)^{-1}\xi \) multiplied by the factor \(
  h_r(E_i,E_i) = g((Z_r(e_i)_\varphi,E_i) = g_{\varphi}(Z_r(e_i),e_i)
  = z_{ii} = 0 \), so \( S^r_{e_i}e_i = 0 \) in these cases.  Moreover,
  \( S^r_\xi\xi = 0 \), and thus we have the claimed vanishing of the
  \( (1,2) \)-trace of \( S^r \).

  For the remaining terms \( \sum_{i=1}^n [e_i',e_i] \) of the \(
  (1,2) \)-trace of \( S^2 \) we choose a different basis \( e_i \).
  Write \( \n = V_{\hol} + \ker\varphi|_{\n} \), in such a way that
  these are \( \h \)-modules whose images in \( \m \) are orthogonal.
  Choose a compatible orthonormal basis \( e_i \) for \( \m \) with \(
  e_i = X_i + \varphi(X_i) \), \( i=1,\dots,n-1 \), such that \( X_i
  \in V_{\hol} \), \( i=1,\dots,k \), and \( X_j \in \ker\varphi|_{\n}
  \), \( j=k+1,\dots,n-1 \) and with \( e_n \) proportional to \( \xi
  \).  Then for \( i=1,\dots,n-1 \) we have
  \begin{equation*}
    [e_i',e_i] = [\varphi(X_i),X_i+\varphi(X_i)] = [\varphi(X_i),X_i].
  \end{equation*}
  This is clearly zero for \( i=k+1,\dots,n-1 \).  For \( i=1,\dots,k
  \), the fact that \( V_{\hol} \) is an \( \h \)-module implies \(
  [\varphi(X_i),X_i] \in V_{\hol} \), so \( [e_i',e_i] =
  \psi^{-1}(\varphi[\varphi(X_i),X_i]) =
  \psi^{-1}[\varphi(X_i),\varphi(X_i)]=0 \).  Finally \( [e_n',e_n] \)
  is proportional to \( [A_0,\xi] = [A_0,A] = 0 \).  Thus in all cases
  \( [e_i',e_i] = 0 \) and \( S^2 \) has no \( \T_1 \) component.

  To see when \( S^2 \) is in \( \T_3 \), consider
  \begin{equation*}
    \begin{split}
      \begin{split}
        S^2_2(B,C)
        &\coloneqq S^2_BC+S^2_CB \\
        &= [C',X_B] + [B',X_C] -
        (Z_r^*(\lambda_BX_C+\lambda_CX_B))_\varphi \eqbreak +
        (h_r(B,C)+h_r(C,B))g(\xi,\xi)^{-1}\xi,
      \end{split}
    \end{split}
  \end{equation*}
  which is proportional to its projection to \( \T_2 \).  For \( S^2
  \) to belong to \( \T_3 \) we need this expression to be zero for
  all \( B \) and \( C \).  First, consider \( C = \xi \) and \( B \)
  orthogonal to \( \xi \), then \( S^2_2(B,\xi) = [\varphi A,X_B] -
  (Z_r^*X_B)_\varphi = 0 \).  This implies that \( g(S^2_2(B,\xi),D) =
  g_\varphi(X_B,[A_0,X_D]) = 0 \), but the representation of \(
  \so(n-1) \) on \( \n \) is faithful, so \( A_0 = 0 \).  Thus \( \m
  \) and hence \( \g \) contains \( \la \) and we may take \( \ls_r =
  \ls \), giving \( A_1 = 0 \) and hence \( \varphi A = 0 \).  We now
  have \( S^2_2(B,C) = [C',X_B] + [B',X_C] \).  Second, suppose \(
  X_B\in\ker\varphi \), then we must have \(
  [C',X_B]=[\varphi(C_{\ls}),X_B]=0 \) for all \( C_{\ls}\in\ls \).
  Thus \( S\in \T_{1+3} \) requires \( \ker\varphi \) to be a trivial
  \( \hol \)-module.  By the proof of Theorem\nobreakspace \ref {thm:holm}, this is
  the case if and only if \( \hol \) is semi-simple and \( \n =
  V_{\hol} + \bR^s \), with the trivial module \( \bR^s \) lying in \(
  \ker \varphi \).  Moreover, in this situation, if \( B,C \) have \(
  X_B,X_C\in V_{\hol} \) then \( S^2_2(B,C) = [\varphi(X_C),X_B] +
  [\varphi(X_B),X_C] = \psi^{-1}([C',B']+[B',C']) = 0 \), so \( S \in
  \T_{1+3} \).

  Furthermore, the \( \T_3 \)-component is non-zero exactly when \(
  \hol \) is non-trivial.  Indeed, in general the \( \T_3 \)-component
  is proportional to
  \begin{equation*}
    \begin{split}
      S^2_3(B,C) &\coloneqq S^2_BC-S^2_CB \\
      &= [B',C]-[C',B] + (Z_r(\lambda_BX_C-\lambda_CX_B))_\varphi.
    \end{split}
  \end{equation*}
  Suppose this tensor~\( S^2_3 \) is zero.  Considering \( B = \xi \)
  and \( C \) orthogonal to \( \xi \), we have \( S^2_3(\xi,C) =
  [\varphi A,C] + Z_r(X_C)_\varphi = [A_0,X_C]_\varphi \), since \(
  [C',A] = 0 = [C',\varphi A] = [C',A_1] \) as \( C' \in \h \).  This
  gives \( A_0 = 0 \), and we may write
  \begin{equation*}
    S^2_3(B,C) =  2[B',C'] + [B',X_C] - [C',X_B].
  \end{equation*}
  For general \( B,C \), the component of \( S^2_3(B,C) \) in \( \h \)
  is \( 2[B',C'] \).  This implies that \( \hol \) is Abelian.
  Finally for \( X_C \in \ker\varphi \) and \( \lambda_C = 0 \), we
  have \( S^2_3(B,C) = [B',X_C] \), so \( \hol \) acts trivially on \(
  \ker\varphi \).  By the proof of Theorem\nobreakspace \ref {thm:holm} we conclude
  that \( \hol = 0 \).  Thus the \( \T_3 \)-component is zero exactly
  when \( \varphi = 0 \).
\end{proof}

\section{Consequences}
\label{sec:consequences}

Our description of splittings via graphs in
Section\nobreakspace \ref {sec:holonomy-algebras} yields the following statement.

\begin{theorem}
  For \( n>1 \), the moduli space of homogeneous tensors on~\( \RH(n)
  \) with fixed scalar curvature, consists of two connected
  components.
\end{theorem}

\begin{proof}
  Any non-zero \( S \) is homotopic to the \( S \) of type \( \T_1 \)
  on \( AN \) via a scaling of \( \varphi \) to~\( 0 \).  So there are
  at most two components in the moduli space.  We need to show that \(
  \{ S = 0 \} \) is a separate component.

  We have \( \RH(n) = \SO(n,1)/\Ort(n) \), \( \SO_0(n,1) = KAN \) and
  Theorem\nobreakspace \ref {thm:structure} tells us that \( A_rN \) acts transitively on
  any homogeneous description of~\( \RH(n) \).  Now \( A_rN \) is
  isomorphic to \( AN \) as a group, and any metric on \( AN \) is
  hyperbolic, indeed the isomorphism may be chosen to be an isometry
  of left-invariant metrics, cf.~\cite{Castrillon-Lopez--GS:RHn}.
  With fixed scalar curvature, we may assume that this is an isometry
  to one fixed choice of left-invariant hyperbolic metric on~\( AN \).
  If \( S \) is a homogeneous tensor on \( \RH(n) \) it gives a
  left-invariant tensor on \( A_rN \) and hence on \( AN \).  The
  equation \( \Nt S = 0 \) may be rewritten as
  \begin{equation}
    \label{eq:LCS}
    \LC S = S.S ,
  \end{equation}
  where \( (S_X.S)_YZ = S_X(S_YZ)-S_{S_XY}Z - S_Y(S_XZ) \), and \( \LC
  \) is the Levi-Civita connection.  On the set of left-invariant
  tensors on~\( AN \), equation~\eqref{eq:LCS} is a set of polynomial
  equations in the components of~\( S \).  We thus see that the set of
  homogeneous tensors for \( \RH(n) \) may be regarded as a real
  algebraic variety~\( \mathscr S \) in \( \bR^n\otimes\Lambda^2\bR^n
  = \bR^N \).  The moduli space is a quotient of \( \mathscr S \) by
  the relation of isomorphism of homogeneous structures; in particular
  tensors \( S \) with different holonomy groups give rise to
  different points of the moduli space.  Once we have shown that \( \{
  S = 0 \} \) is a separate component of \( \mathscr S \), we will
  have that the components of \( \mathscr S \) are preserved by the
  equivalence relation and so give distinct components of the moduli
  space.

  Now for any real algebraic variety~\( \mathscr S \subset \bR^N \)
  and any point~\( S \) in \( \mathscr S \) there is an analytic path
  \( S_t \), \( t\in [0,1] \), with \( S_0 = S \) and \( S_t \ne S \),
  for \( t\in (0,1] \).  Indeed such a path may be taken to be a Nash
  function, see \textcite[Proposition
  8.1.17]{Bochnak-CR:real-algebraic}.  Combining this with
  \cite[Définition et Proposition 2.5.11]{Bochnak-CR:real-algebraic}
  one has that the connected components of \( \mathscr S \) are
  analytically path-connected.

  Suppose \( S_t \) is an analytic path of homogeneous structures with
  \( S_0 = 0 \).  Then equation~\eqref{eq:LCS}, gives that for \( \dot
  S = dS_t/dt |_{t=0} \), we have \( \LC \dot S = \dot S.S_0 + S_0.\dot
  S = 0 \), so \( \dot S \) is parallel for the Levi-Civita
  connection.  But any parallel tensor on \( \RH(n) \) is holonomy
  invariant.  

  The holonomy representation of \( \LC \) on the tangent space of \(
  \RH(n) \) is \( U = \bR^n \) as the standard representation of \(
  \SO(n) \).  The tensor \( \dot S \) lies in
  \begin{equation*}
    U \otimes \Lambda^2 U \cong U \oplus \Lambda^3U \oplus W,
  \end{equation*}
  with \( W \) an irreducible representation of \( \SO(n) \) of
  dimension \( \tfrac13n(n-2)(n+2) \).  This decomposition contains an
  invariant submodule only when \( n=3 \).  So for \( n\ne 3\), we
  conclude that \( \dot S = 0 \).

  We may repeat this argument for the higher derivatives of \( S_t \)
  at \( t=0 \).  When \( n\ne3 \), this gives that \( S_t \) has
  Taylor expansion \( 0 \) around \( t=0 \), and thus that \( S_t \)
  is the constant path.  So for \( n\ne3 \), we have that \( \{ S = 0
  \} \) is a connected component of the moduli space.

  For \( n=3 \), we may argue more directly.  By Theorem\nobreakspace \ref {thm:holm}, the
  holonomy algebras of homogeneous connections on \( \RH(3) \) are \(
  \so(3) \) and \( 0 \), since the only other possibility is \(
  \so(n-1) = \so(2) \), which is Abelian, but has \( 3\dim\so(2) = 3 >
  n-1 = 2 \).  Thus there are only two homogeneous structures on \(
  \RH(3) \), one with \( S=0 \), and the other of type \( \T_1 \) by
  Theorem\nobreakspace \ref {thm:type}.  For structures of type \( \T_1 \), the tensor \( S
  \) is the \( S^1 \) of equation~\eqref{eq:S1}.  The scalar curvature
  of the corresponding metric is determined by \( \norm \xi^2 \), so
  for fixed scalar curvature, there is no path to \( S=0 \), and the
  moduli space again has two components.
\end{proof}

\noindent
Note that the final of the proof part confirms the determination of
homogeneous structures on \( \RH(3) \) by
\textcite{Tricerri-V:homogeneous}.

The proof of the main Theorem\nobreakspace \ref {thm:holm} yields the following information
about the action of the holonomy group.

\begin{corollary}
  Suppose \( \Nt \) is a homogeneous canonical connection on \( \RH(n)
  \) whose holonomy algebra \( \hol \) is not \( \so(n) \).  Then the
  isotropy representation of \( \hol \) on \( \m \) contains at least
  three disjoint modules: the first isomorphic to \( \hol \), the
  second an effective representation of the centre of \( \hol \) and
  the third a one-dimensional trivial module.  If \( \hol \) is
  semi-simple the second module is zero.
\end{corollary}

\begin{proof}
  In the notation of the proof of Theorem\nobreakspace \ref {thm:holm} the three modules
  are \( V_{\hol} \), \( V_1 \) and \( \la_r \).
\end{proof}

Let us regard a geometric structure as any collection of tensors
preserved by some connection, not necessarily torsion-free.  We say
this geometry is homogeneous if can be realised on a reductive
homogeneous space with the connection being the canonical connection.

\begin{corollary}
  Any homogeneous geometry on \( \RH(n) \) that is not invariant under
  the connected isometry group \( SO_0(n,1) \), admits a nowhere
  vanishing parallel vector field. \qed
\end{corollary}

\printbibliography

\bigskip
\begin{small}
  \setlength{\parindent}{0pt} Marco Castrillón López

  Departamento de Geometría y Topología, Facultad de Matemáticas,
  Universidad Complutense de Madrid, Av.\ Complutense s/n,
  28040--Madrid, Spain.

  \textit{E-mail}: \url{mcastri@mat.ucm.es}

  \medskip Pedro Martínez Gadea

  Instituto de Física Fundamental, CSIC, Serrano 113-bis,
  28006--Madrid, Spain.

  \textit{E-Mail}: \url{pmgadea@iff.csic.es}

  \medskip Andrew Swann

  Department of Mathematics, Aarhus University, Ny
  Munkegade 118, Bldg 1530, DK-8000 Aarhus C, Denmark.

  \textit{and}

  CP\textsuperscript3-Origins, Centre of Excellence for Particle
  Physics Phenomenology, University of Southern Denmark, Campusvej 55,
  DK-5230 Odense M, Denmark.

  \textit{E-mail}: \url{swann@imf.au.dk}
\end{small}

\end{document}